\documentclass[10pt,a4paper,reqno]{amsart}%
\usepackage{amsfonts}
\usepackage{amssymb}
\usepackage[utf8]{inputenc}
\usepackage{amsmath}
\usepackage{amsthm}
\usepackage{graphicx}%
\usepackage{enumerate}
\usepackage{xcolor}
\usepackage{color,soul}
\usepackage{tikz}
\usetikzlibrary{patterns}
\usepackage[pdftex,colorlinks]{hyperref}

\newtheorem{thm}{Theorem}[section]

\newtheorem{lem}[thm]{Lemma}
\newtheorem{prop}[thm]{Proposition}

\newtheorem*{assumption*}{\assumptionnumber}
\providecommand{\assumptionnumber}{}
\makeatletter
\newenvironment{assumption}[1]
{%
	\renewcommand{\assumptionnumber}{A.#1}%
	\begin{assumption*}%
		\protected@edef\@currentlabel{A.#1}%
	}
	{%
	\end{assumption*}
}

\theoremstyle{definition}

\theoremstyle{remark}
\newtheorem{rem}[thm]{Remark}

\numberwithin{equation}{section}

\newcommand{\ep}{\varepsilon}

\newcommand{\rz}{\rho_0}

\newcommand{\re}{\rho_\ast}

\newcommand{\rc}{\widehat{\rho}\, }

\newcommand{\R}{\mathbb{R}}				      

\newcommand{\lnum}{ }
\newcommand{\dps}{\displaystyle}	
\newcommand{\intq}{\int_0^T\int_0^1}	
\newcommand{\intw}{\int_0^T\int_{\omega}}	
\newcommand{\intwl}{\int_0^T\int_{\omega'}}	
\newcommand{\into}{\int_0^1}	
\newcommand{\inta}{\int_0^T\int_0^{\alpha'}}
\newcommand{\intb}{\int_0^T\int_{\beta'}^1}
\newcommand{\dom}{(0,T)\times (0,1)}
\newcommand{\domw}{(0,T)\times\omega}

\newcommand{\n}[1]{\|#1\|}
\newcommand{\ny}[1]{\|#1\|_E}
\newcommand{\bu}{\bar{u}}
\newcommand{\bh}{\bar{h}}
\newcommand{\nf}[1]{\left\|#1\right\|_G}
\newcommand{\gu}{\ell\left(\into u\right)}

\newcommand{\bl}{b_\lambda}
\newcommand{\fl}{f_\lambda}

\newcommand{\wei}{e^{2s\varphi}}

\newcommand{\G}[1]{\int_0^T\int_0^1e^{2sA}\left[s\tau a#1_x^2+(s\tau)^3\frac{x^2}{a}#1^2\right]  }
\newcommand{\Ga}[1]{\int_0^{T/2}\int_0^1e^{2sA}\left[s\tau a#1_x^2+(s\tau)^3\frac{x^2}{a}#1^2\right]  }
\newcommand{\Gb}[1]{\int_{T/2}^{T}\int_0^1e^{2sA}\left[s\tau a#1_x^2+(s\tau)^3\frac{x^2}{a}#1^2\right]  }

\definecolor{mycolor}{RGB}{255,251,204}

\usetikzlibrary{shapes,snakes}
\usepackage{environ}

\tikzstyle{mybox} = [draw=yellow, very thick,
rectangle, rounded corners, inner sep=10pt, inner ysep=20pt]
\tikzstyle{fancytitle} =[fill=white, text=red]

\NewEnviron{blc}[1]{\begin{tikzpicture}
	\node [mybox] (box){%
		\begin{minipage}{0.95\textwidth}
		\BODY
		\end{minipage}
	};
	\end{tikzpicture}
}

\begin{document}

	\title[local null controllability for degenerate parabolic equations with nonlocal term]{Local null controllability for degenerate parabolic  equations with nonlocal term}

	\author[R. Demarque]{R. Demarque$^\ast$}
	\address[R. Demarque]{\newline Departamento de Ci\^encias da Natureza,
		Universidade Federal Fluminense,
		Rio das Ostras, RJ, 28895-532, Brazil}
	\email{r.demarque@gmail.com}
	\thanks{$^\ast$ Partially supported by FAPERJ  E-26/111.039/2013 and  Proppi/PDI/UFF}
	\author[J. L\'imaco]{J. L\'imaco}
	\address[J. L\'imaco]{\newline  Departamento de Matem\'{a}tica Aplicada,
		Universidade Federal Fluminense,
		Niter\'{o}i, RJ, 24020-140, Brazil}
	\email{jlimaco@vm.uff.br}
	
	\author[L. Viana]{L. Viana}
	\address[L. Viana]{\newline  Departamento de An\'alise,
		Universidade Federal Fluminense,
		Niter\'{o}i, RJ, 24020-140, Brazil}
	\email{luizviana@id.uff.br}
	\subjclass[2010]{ Primary 35K65, 93B05; Secondary 35K55}
	
	\keywords{Degenerate parabolic equations, controllability, nonlinear parabolic equations, nonlocal term}


	\begin{abstract}
		
		We establish a local null controllability result for following the nonlinear parabolic equation:
		$$u_t-\left(b\left(x,\int_0^1u \   \right)u_x \right)_x+f(t,x,u)=h\chi_\omega,\ (t,x)\in \dom$$
		where $b(x,r)=\ell(r)a(x)$ is a function with separated variables that defines an operator which  degenerates at $x=0$ and has a nonlocal term. Our approach relies on an application of Liusternik's inverse mapping theorem that demands the proof of  a suitable Carleman estimate.
	\end{abstract}
	\maketitle

	\section{Introduction\label{intro}}

\noindent

In this paper we study the null controllability for the degenerate parabolic problem
\begin{equation}
\left\{\begin{array}{l}
u_t-\left(b\left(x,\int_0^1u \   \right)u_x \right)_x+f(t,x,u)=h\chi_\omega,\\
u(t,1)=u(t,0)=0,\\
u(0,x)=u_0(x),
\end{array}\right. \label{prob1} \end{equation}
where $T>0$ is given, $(t,x)\in \dom$,  $u_0\in L^2(0,1)$ and $h\in L^2(\dom)$ is a control that acts on the system through $\omega=(\alpha,\beta)\subset\subset (0,1)$. We also specify some properties of  $b$ and $f$:

\begin{assumption}{1} \label{hyp_a}  Let $\ell:\R\to \R$ be a $C^1$ function with bounded derivative and suppose that $\ell(0)=1$. We also consider $a \in C([0,1])\cap C^1((0,1])$ satisfying $a (0)=0$, $a >0$ on $(0,1]$, $a'\geq 0$ and  \begin{equation}\label{prop_a}
	xa'(x)\leq Ka(x),\ \ \forall x\in [0,1] \mbox{ and some } K\in [0,1).
	\end{equation}
	The function $b:[0,1]\times \R\to \R$  is defined by 
	$$b(x,r)=\ell(r)a(x).$$
\end{assumption}

\begin{rem}
	Let $\alpha \in (0,1)$, then a typical example of function satisfying $\eqref{hyp_a}$ is $a(x)=x^\alpha$. If we define $\beta=\arctan(\alpha)$, then an other example is  $a(x)=x^\alpha\cos(\beta x)$.
\end{rem}


\begin{assumption}{2}\label{hyp_f} Let $f:[0,T]\times [0,1]\times \R\to \R$ be a $C^1$ function  with bounded derivatives such that $f(t,x,0)=0$. We suppose that $$c=c(t,x):=D_3f(t,x,0)\in L^\infty(\dom)$$
	
\end{assumption}

The main goal of this work is to prove that there exists $h\in L^2(\dom)$ such that the associated state $u=u(t,x)$ of \eqref{prob1} satisfies
$$u(T,x)\equiv 0 \mbox{ for any } x\in [0,1],$$
at least if $\|u_0\|_{H_a^1}$ is sufficiently small, where $H_a^1$ is a suitable weighted Hilbert space which will be defined later.

The system considered here yields the operator
$$\left(b\left(\cdot,\int_0^1u\right)u_x\right)_x,$$
which  degenerates at $x=0$ and also have a nonlocal term. Semilinear
nondegenerate equations have been studied extensively in the last forty years, see \cite{fattorini1971exact, fernandez2012null, fernandez2000cost, fursikov1996controllability, lebeau1995controle} for example.

However, there is also a large interest in degenerate operators where degeneracy occurs at the boundary of the space domain. For instance, in order to investigate the Prandtl system for stationary flows, Oleinik et al. \cite{oleinik1999mathematical}
used a transformation introduced by Crocco and reduces the boundary layer system to a single quasilinear equations which is of the degenerate parabolic type. As pointed out by Alabau et al. in \cite{alabau2006carleman}, degenerate operators can also come from probabilistic models, see \cite{feller1952parabolic, feller1954diffusion}. They  have obtained null controllability for the problem \eqref{prob1}  when $b$ does not depend on $\int_0^1u$. Other physical  problems involving degenerate operators can be found in climate science, see for example\cite{floridia2014approximate}.

On the other hand, when $b$ does not depend on $x$, we will have only nonlocal term without degeneration. In this case, the second author et al.  have proved in \cite{fernandez2012null} null controllability   for the following n-dimensional problem
$$\begin{cases}
u_t+B(u(\cdot,t),t)\Delta u=v1_\omega \mbox{ in } \Omega\times (0,T),\\
u(x,t)\equiv \mbox{ on } \partial\Omega \times (0,T),\\
u(x,0)=u_0(x) \mbox{ in } \Omega.
\end{cases}$$
This kind of nonlocal terms have important physical motivations. In that work, the authors listed several examples of real world physical models, namely:

\begin{itemize}
	\item In the case of migration of populations, for instance the bacteria in a container, we may have
	$$B(u(\cdot,t),t)=b\left(\int_\Omega u\right),$$
	where $b$ is a positive continuous function.
	
	\item  In the context of reaction-diffusion systems, it is also frequent to find terms of this kind; the particular case
	$$B(u(\cdot,t),t)=b\left(\langle L, u(\cdot,t)\right\rangle),$$
	where $b$ is a real positive continuous function and $L$ is a continuous linear form on $L^2(\Omega)$, has been investigated for instance by Chang and Chipot \cite{chang2003some}.
	
	\item In the context of hyperbolic equation, terms of the kind
	$$B(u(\cdot,t),t)=b\left(\int_\Omega |\nabla u|^2\right),$$
	appear in the Kirchhoff equations, which arises in nonlinear vibration theory; see for instance \cite{medeiros2002vibrations}.
\end{itemize}

The present work extends the results in \cite{alabau2006carleman} for the case in which the operator degenerates at $x=0$ and also have a nonlocal term. Our approach  is based on the works  \cite{fernandez2012null,clark2013theoretical}. Local null controllability for (\ref{prob1}) will be obtained by applying  the Liusternik's Inverse Mapping Theorem, which can be found in \cite{fursikov1996controllability,yamamoto2003carleman}. More precisely, we will define two Hilbert spaces $E$ and $F$, and a $C^1$ mapping $H:E\to F$ which are related to the null controllability of (\ref{prob1}). The appropriate choice of $E, F$ and $H$ is very meticulous and relies on additional estimates for the solutions of the linearized problem
\begin{equation}\lnum \label{intro2}
\left\{\begin{array}{ll}
u_t-\left(a\left(x\right)u_x \right)_x+c(t,x)u=h\chi_\omega + g, & (t,x)\in \dom, \\
u(t,1)=0,\ u(t,0)=0 & t\in (0,T),\\
u(0,x)=u_0(x).
\end{array}\right.   \end{equation}

The crucial ingredient to assure that $H$ satisfies the hypothesis of Liusternik's Theorem is a Carleman type estimate for  the solutions of the adjoint system of (\ref{intro2}), given by
\begin{equation}\lnum \label{intro3}
\left\{\begin{array}{ll}
v_t+\left(a\left(x\right)v_x \right)_x-c(t,x)v=F(t,x), & (t,x)\in \dom, \\
v(t,1)=0,\ v(t,0)=0 & t\in (0,T).\\
\end{array}\right.   \end{equation}

We observe that the Carleman estimate proved in \cite{alabau2006carleman} as well as that one proved in  \cite{AAC} are not appropriate here. In fact, the estimate obtained in \cite{alabau2006carleman} does not have the observation term in the interior of the domain. In \cite{AAC} the authors dealt with this problem, but they only considered the degeneracy term of the type $x^\alpha$. Our Carleman estimate (Proposition \ref{cor_puel}) is a consequence of two others inequalities. Namely, an extension of that one proved in \cite{AAC} (Proposition \ref{cor_hat}), with the degeneracy term $a=a(x)$ described in assumption \ref{hyp_a}, and  the Hardy-Poincar\'e inequality, obtained in \cite{alabau2006carleman}.

Our main 
result is
the following:
\begin{thm} \label{th_1.1}
	Under the asssumptions on $b$ and $f$, the nonlinear system (\ref{prob1}) is locally null-controllable at any time $T>0$, i.e.,  there exists $\ep>0$ such that, whenever $u_0\in H_a^1$ and $\|u_0\|_{H_a^1}\leq \ep$, there exist a control $h\in L^2(\domw)$ associated to a   state $u=u(t,x)$ satisfying 
	\begin{equation}\label{u(T)}
	u(T,x)=0, \mbox{ for every } x\in [0,1]
	\end{equation}
\end{thm}

This paper is organized as follows. In Section 2, we state some preliminary results. In Section 3, we present a Carleman inequality to the solutions of (\ref{intro3}) and prove the null controllability for the linear system (\ref{intro2}). Finally, Section 4 is devoted to the local null controllability of (\ref{prob1}). 

\section*{Further Comments}

	\begin{enumerate}
		
		\item It would be interesting to know when the same results holds for other boundary conditions. For example, in \cite{alabau2006carleman}, the authors consider two types of degeneracy when $b(x,r)=a(x)$, namely weak and strong degeneracy, each type being associated with its own boundary conditions  at $x=0$. Under this context, we considered, in assumption \eqref{hyp_a}, a weakly degeneracy. It imposed a Dirichlet boundary condition $u(t,0)=1$. However,  if we consider a strong degeneracy, that is, $a \in C^1([0,1])$ satisfying $a (0)=0$, $a >0$ on $(0,1]$ and

		\begin{enumerate}
			\item[($i$)] $xa'(x)\leq Ka(x),\ \ \forall x\in [0,1] \mbox{ and some } K\in [1,2).$
			
			\item[($ii$)] $\begin{cases}
			\exists\theta \in (1,K], \ x\to \frac{a(x)}{x^\theta} \text{ is nondecreasing near } 0, & \text{ if } k>1,\\
			\exists\theta \in (0,1), \ x\to \frac{a(x)}{x^\theta} \text{ is nondecreasing near } 0, & \text{ if } k=1,\\
			\end{cases}$
		\end{enumerate}
		the natural boundary condition to impose at $x=0$ would be of Neumann type 
		$$(au_x)(t,0)=0, \ t\in (0,T).$$
		We believe that  analogous results can be achieved for this  type of degeneracy.
		
		\item Local null boundary controllability is a consequence of our result,  when  $\ell(r)=\operatorname{const.}$ and $b(x,r)=a(x)$. It means that, there exists $\varepsilon>0$ such that if $\n{u_0}_{H_a^1}\leq \varepsilon$, we can take a control $\tilde{h}\in L^\infty(0,T)$ such that the associated solution $u$ to 
		\begin{equation}
		\left\{\begin{array}{ll}
		u_t-\left(a\left(x   \right)u_x \right)_x+f(t,x,u)=0, & \text{in } (0,T)\times (0,1)\\
		u(t,0)=0, u(t,1)=\tilde{h}(t), & \text{in } (0,T)\\
		u(0,x)=u_0(x), & \text{in } (0,1), 
		\end{array}\right. \label{pb-bound} \end{equation}
		satisfies $u(T,x)=0$ in $(0,1)$. The proof of this result is standard: we consider an extended system in  $I_\delta=(0,1+\delta)$ and take $\omega\subset\subset (1,1+\delta)$, so the control $\tilde{h}$ will be the solution of the extended system restricted to $x=1$. When $b(x,r)=a(x)\ell(r)$ boundary controllability is an open question.

		\item A very interesting open question is concerned with global null controllability to \eqref{prob1}, however it does not seem easy. In order to get our main result, we have applied Theorem \ref{liusternik}, which requires the smallness assumptions on the data. Perhaps, to prove a global null controllability result one should use a global inverse mapping theorem, such as Hadamard-Levy Theorem (see \cite{de1994global}), which requires much more complicated estimates. Nevertheless, when $b(x,r)\equiv \operatorname{const.}$ and $f(t,x,u)=f(u)$ satisfies
		$$\lim_{s\to +\infty}\frac{f(s)}{|s|\log^{3/2}(1+|s|)}=0,$$
		the global null controllability holds, see \cite{fernandez2000null}.
	\end{enumerate}

\section{Preliminary Results}

\noindent

In this section we  state some technical results which are necessary to establish Theorem \ref{th_1.1}. At first, we need to introduce some weighted spaces related to the function $a$, namely
\begin{equation*}
H_a^1:= \{  u\in L^2(0,1);\ u\mbox{ is absolutely continuous in } [0,1],\
\sqrt{a}u_x\in L^2(0,1) \mbox{ and } u(1)=u(0)=0\},
\end{equation*}
with  the norm defined by $\|u\|_{H_a^1}^2:=\|u\|_{L^2(0,1)}^2+\|\sqrt{a}u_x\|_{L^2(0,1)}^2$,

\noindent and
$$H_a^2:= \{  u\in H_a^1;\ au_x\in H^1(0,1)\},$$
with the  norm defined by $\|u\|_{H_a^2}^2:=\|u\|_{H_a^1}^2+\|(au_x)_x\|_{L^2(0,1)}^2$.

\begin{rem}
	From inequality (\ref{prop_a}), we can  see that the function $x\mapsto\frac{x^r}{a(x)}$ is nondecreasing on $(0,1]$ for all $r\geq K$ . As a consequence, $\dps x^2/a(x)\leq 1/a(1)$, for all $x\in (0,1]$.
\end{rem} 
In order to deal with the degeneracy of $a$ we need the following  inequality  proved in \cite{alabau2006carleman}.

\begin{prop}[Hardy-Poincar\'e inequality]\label{HP}
	Let $a:[0,1]\to \R$ be a function such that $u\in C([0,1])$, $a(0)=0$ and $a>0$ on $(0,1]$. If there exists $\theta \in (0,1)$ such that  the function $x\mapsto a(x)/x^\theta$ is nonincreasing in $(0,1]$, then there exists a constant $C>0$ such that 
	\begin{equation}\label{HP_ineq}
	\into\frac{a(x)}{x^2}w^2(x)  \leq C\into a(x)|w'(x)|^2  ,
	\end{equation}
	for any function $w$ that is locally absolutely continuous on $(0,1]$, continuous at $0$ and satisfies
	\begin{center}
		$w(0)=0$, and $\dps\into a(x)|w'(x)|^2\   <+\infty$.
	\end{center}
\end{prop}

Let us consider the problem
\begin{equation}\lnum \label{prob2.14}
\left\{\begin{array}{ll}
u_t-\left(a\left(x\right)u_x \right)_x+c(t,x)u=F(t,x)\chi_\omega, & (t,x)\in \dom, \\
u(t,0)=u(t,1)=0, & t\in (0,T),\\
u(0,x)=u_0(x), & x\in (0,1).
\end{array}\right.   \end{equation}
where $a$ satisfies assumption \ref{hyp_a}, $u_0\in L^2(0,1)$ and $F\in L^2(\dom)$. 

In \cite{alabau2006carleman}, the authors use  semigroup theory to obtain the next  well-posedness result for the problem (\ref{prob2.14}).

\begin{prop} If the function $a$ satisfies assumption \ref{hyp_a}, then for all $F\in L^2(\dom)$ and $u_0\in L^2(0,1)$, there exists a unique weak solution $u\in C^0([0,T];L^2(0,1))\cap L^2(0,T;H_a^1)$ of (\ref{prob2.14}). Moreover, if $u_0\in H_a^1$, then 
	$$u\in H^1(0,T;L^2(0,1))\cap L^2(0,T;H_a^2)\cap C^0([0,T];H_a^1),$$
	
	\noindent and there exists $C>0$ such that
	\begin{equation*}
	\sup_{t\in [0,T]}\|u\|_{H_a^1}^2+\int_0^T \|u_t\|^2_{L^2(0,1)}+\int_0^T \|(au_x)_x\|^2_{L^2(0,1)}
	\leq C\left( \|u_0\|^2_{H_a^1}+ \|F\|^2_{L^2(\dom)}\right).
	\end{equation*}
	
\end{prop}

\section{The linearized problem associated to (\ref{prob1})}

\noindent

As we have said before, the local null controllability for (\ref{prob1}) will be obtained from  the global null controllability of its linearized problem. In order to obtain that, we consider 
the problem
\begin{equation}\lnum \label{prob2}
\left\{\begin{array}{ll}
v_t+\left(a\left(x\right)v_x \right)_x-c(t,x)v=F(t,x), & (t,x)\in \dom, \\
v(t,1)=v(t,0)=0 & t\in (0,T),\\
\end{array}\right.   \end{equation}
which is  the adjoint problem of (\ref{prob2.14}) with the forcing term F.

Now we will introduce some functions and notation which will be used from now on. Let $\omega'=(\alpha',\beta')\subset\subset \omega$ and let $\psi:[0,1]\to \R$ be a $C^2 $  function such that 
\begin{align}\label{functions1}
\psi(x):=\begin{cases}
\phantom{-}\int_0^x \frac{y}{a(y)}dy,\ x\in [0,\alpha')\\
-\int_{\beta'}^x \frac{y}{a(y)}dy,\ x\in [\beta',1].
\end{cases}
\end{align}

For $\lambda\geq\lambda_0$ define
\begin{multline}\label{functions}
\theta(t):=\frac{1}{[t(T-t)]^4}, \ \eta(x):=e^{\lambda(|\psi|_\infty+\psi)},\ \sigma(x,t):=\theta(t)\eta(x) \mbox{ and }\\
\varphi(x,t):=\theta(t)(e^{\lambda(|\psi|_\infty+\psi)}-e^{3\lambda|\psi|_\infty}).
\end{multline}


\subsection{Carleman inequalities }
\noindent

The aim of this section is to prove a Carleman type inequality  for solutions of the Problem \ref{prob2}  with weights which do not vanish at $t=0$. To do this, first we will present a result which we adapted from \cite{AAC} and that the proof will be given in the Appendix \ref{appendix}.

\begin{prop}\label{cor_hat}
	There exist $C>0$ and $\lambda_0,s_0>0$ such that every solution  $v$ of (\ref{prob2}) satisfies, for all $s\geq s_0$ and $\lambda\geq \lambda_0$, 
	\begin{equation}
	\intq e^{2s\varphi}\left((s\lambda)\sigma av_x^2+(s\lambda)^2\sigma^2v^2 \right) 
	\leq C\left(\intq e^{2s\varphi}|F|^2\   +(\lambda s)^3\intw e^{2s\varphi}\sigma^3v^2\   \right)\label{carleman_hat}
	\end{equation} 
\end{prop}

Consider a function $m\in C^\infty([0,T])$ satisfying
$$
\left\{\begin{array}{ll}
m(t)\geq t^4(T-t)^4, & t\in \left[0, T/2\right];\\
m(t)= t^4(T-t)^4, & t\in \left[T/2,T\right];\\
m(0)>  0, &
\end{array}
\right.
$$
and define

$$\tau(t):=\frac{1}{m(t)},\ \ \varsigma(x,t):=\tau(t)\eta(x)\ \ \mbox{ and }\ \ A(t,x):=\tau(t)(e^{\lambda(|\psi|_\infty+\psi)}-e^{3\lambda|\psi|_\infty}),$$
where $(t,x)\in [0,T)\times [0,1]$. As usual, we introduce the operators
\begin{align*}
&\Gamma(s,\xi):=\G{\xi} \\
\mbox{and} & \\
&\Gamma_1(s,\xi):=\Ga{\xi}, \ \Gamma_2(s,\xi):=\Gb{\xi}.
\end{align*}

\begin{prop}[Carleman Estimate]\label{cor_puel}
	There exist $C>0$ and $s_0>0$ such that every solution  $v$ of (\ref{prob2}) satisfies, for all $s\geq s_0$, 
	\begin{equation*}
	\intq e^{2sA}\left((s\lambda)\varsigma av_x^2+(s\lambda)^2\varsigma^2v^2\right)   
	\leq C\left(\intq e^{2sA}|F|^2\   +s^3\lambda^3\intw e^{2sA}\varsigma^3v^2\   \right)
	\end{equation*} 
	
\end{prop}

\begin{proof}
	Firstly, we observe that $e^{2s\varphi}\leq e^{2sA}$ and $e^{2s\varphi}\sigma^3 \leq C e^{2sA}\varsigma^3$ for all $(t,x)\in [0,T]\times [0,1]$.

	Secondly, since $\tau=\theta$ and $A=\varphi$ in $[T/2,T]$, Carleman inequality \eqref{carleman_hat} implies
	$$\Gamma_2(s,v)\leq C\left(\intq e^{2A\varphi}|F|^2\   +(\lambda s)^3\intw e^{2sA}\varsigma^3v^2\   \right).$$

	Following the arguments developed in \cite{clark2013theoretical}, Proposition 2.3 page 488, we can prove that
	\begin{equation*}
	\Gamma_1(s,v) \leq C\left(\int_{T/4}^{3T/4}\into e^{2s\varphi}(s\lambda)^2\sigma^2|v|^2+\int_0^{3T/4}\into e^{2sA}|F|^2\right).
	\end{equation*}
	
	Finally, we can use Proposition \ref{cor_hat} and obtain the result.
	
\end{proof}

\subsection{A null controllability result for the linear system}
\noindent

The last goal of this section is to establish a result of global null controllability for the  linear  problem
\begin{equation}\lnum \label{prob3}
\left\{\begin{array}{ll}
u_t-\left(a\left(x\right)u_x \right)_x+c(t,x)u=h\chi_\omega + g, & (t,x)\in \dom; \\
u(t,1)=0,\ u(t,0)=0, & t\in (0,T);\\
u(0,x)=u_0(x),
\end{array}\right.
\end{equation}
where $g\in L^2(\dom)$ , $h\in L^2(\domw)$ and $a$ satisfy the assumption \ref{hyp_a}. In order to state this result, we need to define the weight functions
$$\rho:=e^{-sA}, \ \ \ \rz:=e^{-sA}\varsigma^{-1},\ \ \, \rc:=e^{-sA}\varsigma^{-2},\ \ \, \re:=e^{-sA}\varsigma^{-3}, $$
which satisfy  $\re\leq C\rc\leq C\rz\leq C\rho$ and $\rc^2=\re\rz$.

\begin{prop}\label{prop2.4}
	If $u_0\in H_a^1(0,1)$ and the function $g$ fulfills 
	$$\intq \rho_0^2|g|^2  <\infty,$$
	then  the system (\ref{prob3}) is null-controllable. More precisely, there exists a control $h\in L^2(\domw)$ with associated state $u$ satisfying
	\begin{equation}\label{eq25}
	\intw \rho_\ast^2|h|^2  <+\infty, \ \ \ \intq\rho_0^2|u|^2  <+\infty.
	\end{equation}
	In particular, $u(T,x)\equiv 0$, for all $x\in [0,1]$.
\end{prop}

\begin{proof}
	Firstly, for each $n\in \mathbb N$, we define
	$$A_{n} (t,x) =\frac{ A(T-t)^{4}}{(T-t)^{4} +\frac{1}{n}} \ \ \emph{and} \ \ \varsigma _{n} (t,x) =\frac{\varsigma (T-t)^{4}}{(T-t)^{4} +\frac{1}{n}} ,$$
	where $t\in (0,T)$. We also consider
	$$\displaystyle \rho _{n} = e^{-s A_{n}} \emph{, } \rho _{0 ,n} =\rho _{n} \varsigma _{n}^{-1} \ \ \emph{and} \ \ \rho _{\ast ,n} =\rho _{\ast} m_{n} = \rho \varsigma^{-3} m_{n} ,$$
	where $m_{n} (x)=1$ if $x\in \omega$ and $m_{n} (x)=n$ if $x\not\in \omega$.
	
	For any two functions $u\in L^{2} (\dom)$ and $h\in L^{2} (\domw)$, define
	$$\displaystyle J_{n} (u,h)=\frac{1}{2} \intq \rho _{0 ,n}^{2} |u|^{2}   +\frac{1}{2} \intw \rho _{\ast ,n}^{2} |h|^{2}    \emph{.}$$
	Since each $J_{n}$ is lower semicontinuous, strictly convex and coercive, there exists $(u_{n} ,h_{n}) \in \Lambda:=\{(u,h); h\in L^{2} ((0,T)\times \omega ) \ \ \mbox{and} \ \ (u,h) \ \ \mbox{solves} \ \ (\ref{prob3}) \}$, such that
	$$\displaystyle J_{n} (u_{n} ,h_{n}) = \min \{ J_{n} (u,h); (u,h) \in \Lambda\}.$$
	In this case, $(u_n,h_n)$ satisfies
	\begin{equation}\label{S1}
	\left\{%
	\begin{array}{ll}
	u_{n,t} - (a u_{n,x} )_{x} +c u_{n} = h_{n} \chi _{\omega} +g , & {\ (t,x)\in \dom ,} \\
	u_{n} (t,0) = u_{n} (t,1) = 0 , & {\ t\in (0,T) ,} \\
	u_{n} (0,x) = u_{0 } , & {\ x\in (0,1) ,}
	\end{array}%
	\right. 
	\end{equation}
	and Lagrange's Principle assures the existence of a function $p_{n}$ solving
	\begin{equation}\label{S2}
	\left\{\begin{array}{ll}
	-p_{n,t} - (a p_{n,x} )_{x} +cp_{n} = -\rho _{0 ,n}^{2} u_{n} , & \  (t,x)\in \dom ,\\
	p_{n} (t,0)= p_{n} (t,1) = 0  , & {\ t\in (0,T) ,} \\
	p_{n} (T,x)=0 , & {\ x\in (0,1) ,} \\
	p_{n} (t,x) = \rho_{\ast ,n}^{2} h_{n} , & \ (t,x)\in (0,T) \times \omega \emph{.}
	\end{array}\right.
	\end{equation}

	By standard arguments, we can prove that $J_{n} (u_{n} ,h_{n} ) \leq C \sqrt{J_{n} (u_{n} ,h_{n} )}$ for all $n\in \mathbb N$, i.e., \linebreak $(J_{n} (u_{n} ,h_{n} ) )_{n=1}^{\infty}$ is a numerical bounded sequence.
	
	Since $\rho_{0,n}^2,\rho_{\ast,n}^2\geq C$, we deduce that
	$$\intq |u_n|^2  +\intw |h_n|^2  \leq CJ_n(u_n,h_n)\leq C.$$
	It means that there exist $u\in L^2(\dom)$ and $h\in L^2(\domw)$ such that, up to  subsequences, we have
	$$u_n\rightharpoonup u \ \ \mbox{ and } \ \ h_n\rightharpoonup h \ \ \mbox{ in } L^2(\dom). $$
	From this, we take
	\begin{equation}\label{wc}
	\rho_{0,n}u_n\rightharpoonup \rho_0 u \ \ \mbox{ and } \ \ \rho_{\ast,n}h_n\rightharpoonup \rho_{\ast}h \ \ \mbox{ in } L^2(\dom).
	\end{equation} 
	
	Consequently, passing to limits as $n\to +\infty$, we conclude that $(u,h)$ solves (\ref{prob3}).

	Furthermore, (\ref{eq25}) follows from  (\ref{wc}). This establishes the result.

\end{proof}

\section{Main Result}

\noindent

This section is devoted to prove Theorem \ref{th_1.1}. As we have said in the introduction, our approach relies on setting an appropriate mapping  $H:E\to F$ for which we will apply Liusternik's Theorem.

\subsection{Functional Spaces}

\noindent

Consider the Hilbert spaces  $E$ and $F$  
\begin{multline*}
E:= \bigg\{ (u,h); u\in L^2(\dom), h\in L^2(\domw),
u_t,u_x,(a(x)u_x)_x, \re h\in L^2(\domw),\\
\rz u, \rz( u_t-(a(x)u_x)_x-h\chi_\omega)\in L^2(\dom)
u(\cdot,1)\equiv u(\cdot,0)\equiv 0,  u(0,\cdot)\in H_a^1 \bigg\}
\end{multline*}
and
$$F:=G\times H_a^1, \mbox{ where } G:=\left\{ g\in L^2(\dom);\ \rz g\in L^2(\dom) \right\},$$
with the norms
\begin{equation*}
\|(u,h)\|_E^2:= \intq\rz^2|u|^2  +\intw \re^2|h|^2   
\intq\rz^2|u_t-(a(x)u_x)_x-h\chi_\omega|^2  +\|u(\cdot,0)\|_{H_a^1}^2
\end{equation*}
and
\begin{equation*}
\|(g,v)\|_F^2:=\intq \rz^2g^2+\intq v^2+\intq av_x^2.
\end{equation*}

Next,  we will state a crucial result that will allow us to set the mapping $H$. Its proof is a consequence of two lemmas which will be established right below.

\begin{prop}\label{lema3.2} There exists $C>0$ such that 
	$$\intq\re^2(|u_t|^2+|(a(x)u_x)_x|^2)  \leq C\|(u,h)\|_E^2,$$
	for all $(u,h)\in E$.
\end{prop}

\begin{lem}\label{prop2.5}
	Assume the hypothesis of Theorem \ref{prop2.4}. Then
	\begin{equation*}
	\intq \rc^2a(x)|u_x|^2  
	\leq C\left(\intq\rz^2|u|^2  +\intw\re^2|h|^2  +\intq\rz^2g^2   
	+\|u_0\|_{H^1_a}^2\right)
	\end{equation*}
	
	\begin{proof}
		
		Multiplying the PDE in (\ref{prob3}) by $\rc^2u$, integrating in $[0,1]$ and using the two relations
		
		$$ \frac{1}{2}\frac{d}{dt}\into\rc^2u^2=\int_0^1\rc^2u_tu+\into\rc\rc_tu^2$$
		and
		\begin{align*}
		\into\rc^2(au_x)_xu\    & =-2\into \rc\rc_x auu_x- \into\rc^2au_x^2,
		\end{align*}
		we obtain
		\begin{align}\label{myeq2}
		\frac{1}{2}\frac{d}{dt}\into\rc^2u^2 \   +\into\rc^2au_x^2 & =- \into\rc^2cu^2+\into\rc^2uh\chi_\omega +\into\rc^2 gu+\into\rc\rc_tu^2-2\into\rc\rc_xauu_x\nonumber\\
		& = I_1+I_2+I_3+I_4+I_5.
		\end{align}

		Now, using $\re\leq C\rc\leq C\rz\leq C\rho$ and $\re\rz=\rc^2$, we obtain
		\begin{align*}
		& I_1\leq C \into\rz^2|u|^2  ,\\
		& I_2\leq C\left(\frac{1}{2}\int_0^1\re^2|h\chi_\omega|^2\   +\frac{1}{2}\int_0^1\rz^2|u|^2\   \right),\\
		& I_3\leq C\left(\frac{1}{2}\int_0^1\rz^2|g|^2\   +\frac{1}{2}\int_0^1\rz^2|u|^2\   \right).
		\end{align*}

		Let us estimate $I_4$. First, we will rewrite $A$ as 
		$A(t,x)=\varsigma(t,x) \bar{\eta}(x)$, where  $\bar{\eta}(x):=(e^{\lambda(|\psi|_\infty+\psi)}-e^{2\lambda|\psi|_\infty})/\eta(x)$. Second, note that
		$$\rc\rc_t=-se^{-2sA}\bar{\eta}(x)\varsigma^{-4}\varsigma_t-2e^{-2sA}\varsigma^{-5}\varsigma_t.$$
		
		Then, for all $t\in [0,T]$,
		\begin{align*}
		|\rc\rc_t| \leq C\rz^2|\varsigma^{-2}+\varsigma^{-3}|\left|\varsigma_t\right| \leq C\rz^2,
		\end{align*}
		whence
		$$I_4\leq C\into\rz^2|u|^2  .$$
		
		Now, using
		\begin{align*}
		\rc^2_xau^2 &\leq Ce^{-2sA}\varsigma^{-2}\left|\varsigma^{-2}+\varsigma^{-4}\right| |\varsigma_x^2|au^2\leq c\rz^2u^2,
		\end{align*}
		we obtain
		\begin{equation*}
		I_5 \leq 2\into|\rc\sqrt{a}u_x||\rc_x\sqrt{a}u| \leq \frac{1}{2}\into\rc^2au_x^2+2\into \rc_x^2au^2  \leq \frac{1}{2}\into\rc^2au_x^2+C\into \rz^2u^2.
		\end{equation*}
		Hence, (\ref{myeq2}) gives us
		$$\frac{d}{dt}\into\rc^2|u|^2  +\into\rc^2a|u_x|^2  \leq C\left(\into\rz^2|u|^2  \into\re^2|h\chi_\omega|^2   \into\rz^2|g|^2  \right).$$ 
		Integrating in time, the result follows.

	\end{proof}
	
\end{lem}

\begin{lem}\label{prop2.6}
	Assume the hypothesis of Proposition \ref{prop2.4} and suppose that $h$ and $u$ satisfy (\ref{prob3}) and (\ref{eq25}). Then
	\begin{equation*}
	\intq\re^2u_t^2   +\intq\re^2|(a(x)u_x)_x|^2  
	\leq C\left(\intq\rz^2u^2  +\intw\re^2h^2  +\intq\rz^2g^2   
	+\|u_0\|_{H_a^1}^2\right).
	\end{equation*}
\end{lem}

\begin{proof}
	In the first step,  we will estimate the first term of left side of the inequality. Multiplying the PDE in (\ref{prob3}) by $\re^2u_t$ and integrating in $[0,1]$ we have
	
	\begin{align}\label{myeq3}
	\int_0^1\re^2u_t^2 &=  \int_0^1\re^2u_th\chi_\omega   +\into\re^2g u_t -\into c(t,x)\re^2uu_t  +\into\re^2(au_x)_xu_t\   
	\nonumber \\
	& =I_1+I_2-I_3+I_4.
	\end{align}
	
	Using Young's inequality with $\ep$ and $\re\leq C\rc\leq C\rz\leq C\rho$ we obtain
	\begin{align*}
	I_1  &\leq \into \re^2|h\chi_\omega||u_t|\   \leq \ep\into\re^2|u_t|^2  +\frac{1}{4\ep}\into\re^2|h\chi_\omega|^2,\\
	I_2 & \leq \into\re^2|gu_t| \leq \ep\into\re^2|u_t|^2  +\frac{1}{4\ep}\into\re^2|g|^2  \leq \ep\into\re^2|u_t|^2  +\frac{C}{4\ep}\into\rz^2|g|^2 
	\end{align*}
	and
	\begin{align*}
	-I_3\leq \into |c(t,x)|\re^2|uu_t|   & \leq  C\left(\ep\into\re^2u_t^2  +\frac{1}{4\ep}\into\rz^2u^2   \right).
	\end{align*}
	
	Now, integrating $I_4$ by parts, we can see that
	\begin{align}\label{myeq37} 
	I_4& =\left.\re^2au_xu_t\right\vert_{x=0}^{x=1} -\into(\re^2u_t)_xau_x =-2\into\re(\re)_xau_tu_x-\frac{1}{2}\frac{d}{dt}\into\re^2a u_x^2 +\frac{1}{2}\into(\re^2)_tau_x^2\nonumber\\
	& =-2I_{41}-\frac{1}{2}\frac{d}{dt}\into\re^2a u_x^2\   +\frac{1}{2} I_{42}.
	\end{align}
	
	Hence,
	\begin{equation}\label{myeq4}
	\int_0^1\re^2|u_t|^2\   +\frac{1}{2}\frac{d}{dt}\into\re^2a |u_x|^2\    =I_1+I_2-I_3-2I_{41}+\frac{1}{2}I_{42}.
	\end{equation}
	Since 
	$$|\re(\re)_xau_xu_t|\leq C|\re u_t||\rc\sqrt{a}u_x|$$
	and
	$$|(\re^2)_t|= 2|\re(\re)_t|\leq C\rc^2,$$
	we have
	$$I_{41}\leq \frac{1}{4}\into\re^2u_t^2+C\into\rc^2au_x^2$$
	and 
	\begin{equation*}\label{myeq8}
	I_{42} \leq C\into\rc^2au_x^2  .
	\end{equation*}
	For the estimates of $I_1,I_2,I_3, I_{41}$ and $I_{42}$ in the (\ref{myeq4}) we obtain
	
	\begin{equation*}
	\int_0^1\re^2u_t^2\   +\frac{1}{2}\frac{d}{dt}\into\re^2a u_x^2\      \leq C\left(\into\re^2|h\chi_\omega|^2  +\into\rz^2g^2 \right.
	\left.+\into \rz^2u^2  +\into\rc^2au_x^2  \right),
	\end{equation*}
	which implies
	\begin{equation}\label{myeq9}
	\intq\re^2u_t^2   \leq C\left(\intq\rz^2u^2  +\intw\re^2h^2  \right. 	\left.+\intq\rz^2g^2   
	+\|u_0\|_{H_a^1}^2\right).
	\end{equation}

	In the second part, we must estimate the term $\intq\re^2|(au_x)_x|^2  $. Multiplying the PDE in (\ref{prob3}) by $-\re^2(au_x)_x$ and integrating in $[0,1]$, we take
	\begin{align*} 
	\into\re^2|(au_x)_x|^2   & =-\into\re^2h\chi_\omega(au_x)_x  - \into\re^2g(au_x)_x  +\into c(t,x)\re^2u(au_x)_x  +\into\re^2u_t(au_x)_x  \nonumber\\
	& =-J_1-J_2+J_3+I_4.
	\end{align*}
	
	Again, applying Young's inequality with $\ep$, we obtain
	$$J_1\leq \into\re^2|h\chi_\omega||(au_x)_x|  \leq \ep\into\re^2|(au_x)_x|^2  +\frac{1}{4\ep}\into\re^2|h\chi_\omega|^2  .$$
	$$J_2\leq \into\re^2|g||(au_x)_x|  \leq \ep\into\re^2|(au_x)_x|^2  +\frac{1}{4\ep}\into\rz^2g^2  .$$
	\begin{align*}
	J_3 & \leq C\left(\ep\into\re^2|(au_x)_x|^2  +\frac{1}{4\ep}\into\rz^2u^2  .\right)
	\end{align*}
	
	Recalling the  identities \eqref{myeq37} and \eqref{myeq9}, we have
	\begin{equation*}
	\into\re^2|(au_x)_x|^2  +\frac{1}{2}\frac{d}{dt}\into\re^2a|u_x|^2  \leq C\left(\into\re^2|h\chi_\omega|^2  +\into\rz^2|k|^2  \right.
	\left.+\into\rz^2|u|^2  +\into\rc^2a|u_x|^2  \right)
	\end{equation*}
	Integrating in time and recalling  Lemma (\ref{prop2.5}), we conclude the proof.
\end{proof}

In order to establish the last result of this section, we need to prove a technical lemma.


\begin{lem}\label{lema3.1} 
	Let $\beta(x)=e^{\lambda(|\psi|_\infty+\psi)}-e^{3\lambda|\psi|_\infty}$ and $\bar{\beta}=\dps\max_{x\in [0,1]}\beta(x)$. 
	There exists $s> 0$  such that if  $\dps s\bar{\beta}<M<0$, then
	$$\sup_{t\in[0,T]}\left\{e^{-2M/m(t)} \left(\int_0^1u\   \right)^2\right\}\leq C\|(u,h)\|_E^2,$$
	for all $(u,h)\in E$.
\end{lem}

\begin{proof}
	Firstly, consider  $(u,h)\in E$ and  $q:[0,T]\to \R$ given by
	$$q(t):=e^{-M/m(t)}\int_0^1u(t,x).$$
	
	\noindent \textbf{Claim 1:} There exists $s>0$ and $C>0$ such that $e^{-2M/m(t)}\leq C\re^2 $.

	Indeed, take $k>0$ and $C>0$ satisfying $e^{-k/x}\leq Cx^{3}$ for all $x> 0$. In this case,  we have 
	$$e^{-k/m(t)}\leq Cm^{3}(t),\ \mbox{ for all }\ t\in [0,T].$$
	
	Now, taking $s>0$ such that $2s(\bar{\beta}-\beta)>k$, we obtain
	$$e^{-2M/m(t)+2sA}=e^{(-2M+2s\beta)/m(t)}<e^{-k/m(t)}<Cm^3(t),\ \mbox{ for all } t\in [0,T],$$
	whence we deduce that
	$$e^{-2M/m(t)}\leq Ce^{-2sA}\tau^{-3}\leq C\re^2 ,\ \mbox{ for all } t\in [0,T].$$
	
	\noindent \textbf{Claim 2:} $\|q\|_{H^1(0,T)}\leq C\|(u,h)\|_E$.

	In fact, since $\re^2\leq C\rz^2$, Claim 1  and Proposition \ref{lema3.2} imply
	$$\|q\|^2_{L^2(0,T)} \leq\intq e^{-2M/m(t)} u^2\   \leq C\intq \re^2u^2  \leq C\|(u,h)\|_E^2 $$
	and
	\begin{align*}
	\|q'\|^2_{L^2(0,T)}  &= \int_0^T \left|\frac{Mm'(t)}{m^2(t)}q(t)+e^{-M/m(t)}\into u_t\right|^2 \leq C\int_0^T\frac{M^2}{m^4(t)}|q(t)|^2+C\intq e^{-2M/m(t)}u_t^2  \\
	& \leq C\intq\frac{M^2e^{-2M/m(t)}}{m^4(t)}u^2  +C\intq e^{-2M/m(t)}u_t^2  \leq C\intq \rz^2|u|^2  +C\intq\re^2|u_t|^2\\
	& \leq C\|(u,h)\|_E^2.
	\end{align*}
	This ends the proof of Claim 2.
	
	As a conclusion, the proof comes from Claim 2 and the continuous embedding $H^1(0,T)\hookrightarrow C(0,T)$.
\end{proof}

\begin{prop}\label{prop3.1}
	There exists $C>0$ such that
	$$\intq \rz^2\left(\int_0^1\bu\ \right)^2|(au_x)_x|^2\leq C\|(u,h)\|_E^2\|(\bu,\bh)\|_E^2,$$	
	for any $(u,h), (\bu,\bh)\in E$.
\end{prop}

\begin{proof} In fact, since
	$\tau^4\leq\frac{3}{-2M^4} e^{-2M/m(t)}$, applying Proposition \ref{lema3.2} and Lemma \ref{lema3.1} we have
	\begin{align*}
	& \intq \rz^2\left(\int_0^1u\   \right)^2\left|\left(au_x\right)_x\right|^2   \leq C\intq \tau^4\left(\int_0^1\bu\   \right)^2\re^2\left|\left(au_x\right)_x\right|^2   \\
	& \leq C\dps\sup_{t\in[0,T]}\left\{e^{-2M/m(t)}\left(\int_0^1\bu\   \right)^2\right\}\intq \re^2\left|\left(au_x\right)_x\right|^2   \leq C\|(u,h)\|_E^2\|(\bu,\bh)\|_E^2.
	\end{align*}
\end{proof}

\subsection{Local Null Controllability  for the nonlinear system}

\noindent

Consider the mapping $H:E\to F$, defined by
\begin{equation}
H(u,h) :=(H_1(u,h),H_2(u,h)):= \left(u_t-\left(b\left(x,\int_0^1u\   \right)u_x\right)_x+f(t,x,u)-h\chi_\omega, u(\cdot,0)\right).\label{H}
\end{equation}

Our goal is to prove that $H$ verifies the hypothesis of the following version of Liusternik's Theorem.

\begin{thm}[Liusternik]
	Let $E$ and $F$ be two Banach spaces, $H: E\to F$ a $C^1$ mapping and $\eta_0=H(0)$. If $H'(0):E\to F$ is onto, then there exist $\ep>0$ and $\tilde{H}:B_\ep(\eta_0)\subset F\to E$ such that
	$$H(\tilde{H}(\xi))=\xi,\ \forall \xi\in B_\ep(\eta_0),$$
	that is, $\tilde{H}$ is a right inverse of $H$.
\end{thm}

\begin{lem}\label{lema3.3}
	Let $H:E\to F$ be the mapping defined by (\ref{H}). Then $H$ is well defined.
\end{lem}
\begin{proof}
	Indeed, given $(u,h)\in E$, we already have $H_2(u,h)=u(0,\cdot)\in H^1_a$. Moreover,
	\begin{align*}
	& \intq\rz^2|H_1(u,h)|^2  = \intq \rz^2\left|u_t-\left(b\left(x,\int_0^1u\   \right)u_x\right)_x+f(t,x,u)-h\chi_\omega\right|^2   \\
	& \leq 3\intq \rz^2\left|u_t-\left(b\left(x,0\right)u_x\right)_x-h\chi_\omega\right|^2   + 3\intq \rz^2\left|\left(\left[b\left(x,\int_0^1u\   \right)-b\left(x,0\right)\right]u_x\right)_x\right|^2   \\
	&\phantom{ \leq  } +3 \intq \rz^2\left|f(t,x,u)\right|^2   \\
	& \leq 3A_1+3A_2+3A_3.
	\end{align*}
	
	The definition of the space $E$ gives us $A_1   \leq \|(u,h)\|_E^2$. Also, the assumption \ref{hyp_f} implies
	$$A_3=\intq\rz^2|f(t,x,u)-f(t,x,0)|^2\leq C\intq \rz^2|u|^2   \leq \|(u,h)\|_E^2.$$
	
	It remains to analyze $A_2$. Since $\ell$ is Lipschitz-continuous and applying Proposition \ref{prop3.1} we have
	\begin{equation*}
	A_2  =\intq \rz^2\left|\left[\ell\left(\int_0^1u\   \right)-\ell(0)\right]\left(a(x)u_x\right)_x\right|^2   \leq C\intq \rz^2\left(\int_0^1u\   \right)^2\left|\left(a(x)u_x\right)_x\right|^2  \leq C\|(u,h)\|_E^4.
	\end{equation*}
	
	In this case, we also have $H_1(u,h)\in G$, which completes the proof.
\end{proof}

\begin{lem}\label{lema3.4}
	The mapping $H$ is of class $C^1$. 
\end{lem}

\begin{proof}It is clear that $H_2\in C^1$. We will prove that $H_1$ has a continuous Gateaux derivative on $E$.

	For $(u,h),(\bu,\bh)\in E$ and $\lambda>0$, set
	\begin{align*}
	& b:=b\left(x,\into u\right), & &\bl:= b\left(x,\into(u+\lambda\bu)\right),&   & b_2:=D_2b\left(x,\into u\right),  \\
	& f:=f(t,x,u), & & \fl:=f(t,x,u+\lambda \bu),& & f_3:=D_3f(t,x,u).
	\end{align*}

	\noindent\textbf{Claim 1:} Given $(u,h)\in E$, the linear mapping $L:E\to F$, defined by
	$$L(\bu,\bh):=\bu_t-\left(b_2\cdot\left(\into \bu\right)u_x+b\bu_x\right)_x+f_3\bu-\bh\chi_\omega,$$
	is the Gateaux derivative of $H_1$ at $(u,h)\in E$.
	
	Indeed, for any $(\bu,\bh)\in E$, we have
	\begin{align*}
	& \nf{\frac{1}{\lambda}\left(H_1(u+\lambda \bu,h+\lambda \bh)-H_1(u,h)\right)-L(\bu,\bh)}\\
	&=\nf{\bu_t-(\bl \bu_x)_x-\frac{1}{\lambda} \left((\bl-b)u_x\right)_x+\frac{1}{\lambda}(\fl-f)-\bh\chi_\omega-L(\bu,\bh)}\\
	&\leq \nf{\left((\bl-b)\bu_x\right)_x}+\nf{\left(\left[\frac{1}{\lambda}(\bl-b)-b_2\left(\into \bu\right)\right]u_x \right)_x}+\nf{\frac{1}{\lambda}(\fl-f)-f_3\bu}\\
	&= B_1+B_2+B_3.
	\end{align*}
	We must prove that $B_i\to 0$, as $\lambda \to 0$, for all $i=1,2,3$.
	
	Firstly, recalling the assumption \ref{hyp_f}, we apply mean value 
	and Lebesgue's theorems to obtain
	\begin{align*}
	B_3^2 & \leq \intq \rz^2\left|(D_3f(t,x,u_\lambda^\ast)-D_3f(t,x,u))\bu\right|^2\to 0,
	\end{align*}
	as $\lambda \to 0$.
	
	Secondly, the assumption \ref{hyp_a} and Proposition \ref{prop3.1}, we have
	\begin{align*}
	B_1^2  &=\intq \rz^2\left|\ell\left(\into (u+\lambda\bu)\right)-\gu\right|^2|(au_x)_x|^2 \leq C\intq \rz^2\left(\into \lambda \bu\right)^2|(au_x)_x|^2\\
	&\leq C\lambda^2\ny{(\bu,\bh)}^2\ny{(u,h)}^2\to 0,
	\end{align*}
	as $\lambda \to 0$.
	
	In a similar way, using the assumption \ref{hyp_a} and Mean Value Theorem, for each $\lambda>0$, there exists  $\xi_{\lambda}\in \R$, between $\into u$ and $ \into (u+\lambda \bu)$, such that
	\begin{align*}
	B_2^2& =\intq \rz^2\bigg|\frac{1}{\lambda}\left[\ell\left(\into (u+\lambda\bu)\right)-\gu\right]
	-\ell'\left(\into u\right)\left(\into \bu\right)\bigg|^2\left|\left(au_x\right)_x\right|^2\\
	&=\intq \rz^2\left|\ell'\left(\xi_\lambda\right)\left(\into \bu\right)-\ell'\left(\into u\right)\left(\into \bu\right)\right|^2\left|\left(au_x\right)_x\right|^2\\
	& =\intq \rz^2\left|\ell'\left(\xi_\lambda\right)-\ell'\left(\into u\right)\right|^2\left(\into \bu\right)^2\left|\left(au_x\right)_x\right|^2\to 0,
	\end{align*}
	as $\lambda \to 0$. Thus, we have concluded the Claim 1.
	
	\noindent\textbf{Claim 2:} The Gateaux derivative $H'_1:E\to \mathcal{L}(E,G)$ is continuous.

	Take $(u,h)\in E$. Let $((u^{n} ,h^{n}) )_{n=1}^{\infty}$ be a sequence such that $\| (u^{n} ,h^{n} )-(u,h) \| _{E} \rightarrow 0$. We will prove that $\| H'_1 (u^{n} ,h^{n} )-H'_1 (u,h) \| _{\mathcal L (E,G)} \rightarrow 0$. In fact, consider $(\bu ,\bh) $ on the unit sphere of $E$. Since
	\begin{align*}
	\displaystyle H_{1}^{\prime} (u,h) (\bu ,\bh ) =& \bu_{t} - \left( D_{2} b \left( x,\int_{0}^{1} u \right) \left( \int_{0}^{1} \bu \right) u_{x} + b \left( x,\int_{0}^{1} u \right) \bu_{x} \right) _{x} \\
	&+ D_{3} f(t,x,u) \bu + \bh \chi _{\omega} , 
	\end{align*}
	and
	\begin{align*}
	\displaystyle H_{1}^{\prime} (u^{n} ,h^{n} ) (\bu ,\bh ) = & \bu_{t} - \left( D_{2} b \left( x,\int_{0}^{1} u^{n} \right) \left( \int_{0}^{1} \bu \right) u_{x}^{n} + b \left( x,\int_{0}^{1} u^{n} \right) \bu_{x} \right) _{x} \\
	&+ D_{3} f(t,x,u^{n} ) \bu + \bh \chi _{\omega},
	\end{align*}
	we obtain
	\begin{align*}
	&\displaystyle \| (H'_1 (u^{n} ,h^{n} ) - H^{\prime}_1 (u,h) )(\bu ,\bh)\| _{G}^{2}\\
	&\qquad {} \leq C\int _{Q} \rho _{0}^{2} \left(\int_{0}^{1} \bu \right) ^{2} \left[ \ell^{\prime} \left( \int_{0}^{1} u \right) \right] ^{2} |[a(u-u^{n} )_{x} ]_{x} |^{2} \\
	&\phantom{\leq } \qquad {} +C \intq \rho _{0}^{2} \left(\int_{0}^{1} \bu \right) ^{2} \left| \ell^{\prime} \left( \int_{0}^{1} u \right) - \ell^{\prime} \left( \int_{0}^{1} u^{n} \right) \right| ^{2} |(au^{n}_{x} )_{x} |^{2} \\
	&\phantom{\leq } \qquad {} +C \intq \rho _{0}^{2} \left| \ell\left( \int_{0}^{1} u \right) - \ell
	\left( \int_{0}^{1} u^{n} \right) \right| ^{2} |(a\bu_{x} )_{x} |^{2} \\
	&\phantom{\leq } \qquad {}+C \intq \rho _{0}^{2} |\bu |^{2} |D_{3} f(t,x,u^{n} ) -D_{3} f(t,x,u) |^{2} \\
	&\qquad {}\leq C\int _{Q} \rho _{0}^{2} \left(\int_{0}^{1} \bu \right) ^{2} |[a(u-u^{n} )_{x} ]_{x} |^{2} \\
	&\phantom{\leq }\qquad {} +C \intq \rho _{0}^{2} \left(\int_{0}^{1} \bu \right) ^{2} \left| \ell^{\prime} \left( \int_{0}^{1} u \right) - \ell^{\prime} \left( \int_{0}^{1} u^{n} \right) \right| ^{2} |(au^{n}_{x} )_{x} |^{2} \\
	&\phantom{\leq }\qquad {} +C \intq \rho _{0}^{2} \left| \int_{0}^{1} (u^{n} -u) \right| ^{2} |(a\bu_{x} )_{x} |^{2} \\
	&\phantom{\leq } \qquad {} +C \intq \rho _{0}^{2} |\bu |^{2} |D_{3} f(t,x,u^{n} ) -D_{3} f(t,x,u) |^{2} \\
	&\qquad {} :=C (I _{1} + I _{2} + I _{3} + I _{4} ) \emph{.}
	\end{align*} 
	
	Due to Proposition \ref{lema3.2}, we have
	\begin{equation*}
	\displaystyle  I _{1} 
	= \int _{Q} \rho _{0}^{2} \left(\int_{0}^{1} \bu \right) ^{2} |[a(u-u^{n} )_{x} ]_{x} |^{2} \leq C \| (\bu ,\bh )\| _{E}^{2} \intq \rho _{\ast}^{2} | (a(u_{n} -u)_{x} )_{x} |^{2} 
	\leq C \| (u^{n} ,h^{n} ) - (u,h)\|_{E}^{2} \emph{.} 
	\end{equation*}
	
	Next, 
	\begin{align*}
	\displaystyle  I _{2}
	&=\intq \rho _{0}^{2} \left(\int_{0}^{1} \bu \right) ^{2} \left| \ell^{\prime} \left( \int_{0}^{1} u \right) - \ell^{\prime} \left( \int_{0}^{1} u^{n} \right) \right| ^{2} |(au^{n}_{x} )_{x} |^{2} \\
	&\leq C \intq \rho _{0}^{2} \left(\int_{0}^{1} \bu \right) ^{2} \left| \ell^{\prime} \left( \int_{0}^{1} u \right) - \ell^{\prime} \left( \int_{0}^{1} u^{n} \right) \right| ^{2} |(a(u^{n}-u)_{x} )_{x} |^{2}  \\
	&+C \intq \rho _{0}^{2} \left(\int_{0}^{1} \bu \right) ^{2} \left| \ell^{\prime} \left( \int_{0}^{1} u \right) - \ell^{\prime} \left( \int_{0}^{1} u^{n} \right) \right| ^{2} |(au_{x} )_{x} |^{2}  \\
	&:= I _{21} +  I _{22} \emph{.}
	\end{align*}
	
	Using assumption \ref{hyp_a} and applying Proposition \ref{lema3.2} again, we have
	\begin{align*}
	\displaystyle  I _{21}&= \intq e^{-s A} \tau ^{4} \left(\int_{0}^{1} \bu \right) ^{2} \bigg| \ell^{\prime} \left( \int_{0}^{1} u \right) 
	- \ell^{\prime} \left( \int_{0}^{1} u^{n} \right) \bigg| ^{2}
	\rho _{\ast}^{2} |(a(u^{n}-u)_{x} )_{x} |^{2} \\
	&\leq C \| (\bu ,\bh )\| _{E}^{2} \| (u^{n} ,h^{n} ) - (u,h)\| _{E}^{2}= C \| (u^{n} ,h^{n} ) - (u,h)\| _{E}^{2}.
	\end{align*} 
	Likewise,
	\begin{align*}
	\displaystyle  I _{22} 
	&=\intq \rho _{0}^{2} \left(\int_{0}^{1} \bu \right) ^{2} \left| \ell^{\prime} \left( \int_{0}^{1} u \right) - \ell^{\prime} \left( \int_{0}^{1} u^{n} \right) \right| ^{2} |(a(u^{n}-u)_{x} )_{x} |^{2} \\
	&\leq C\|(\bu ,\bh )\| _{E}^{2} 
	\intq \left| \ell^{\prime} \left( \int_{0}^{1} u \right) - \ell^{\prime} \left( \int_{0}^{1} u^{n} \right) \right| ^{2} \rho _{\ast}^{2} |(a u_{x} )_{x} |^{2} \to 0,
	\end{align*}
	as $n\to +\infty$, where we have used Lebesgue's Theorem.
	
	Analogously, using assumption \ref{hyp_f}, Lemma \ref{lema3.1} and Proposition \ref{lema3.2}, we have
	\begin{align*}
	\displaystyle
	I _{3}
	&=\intq \rho _{0}^{2} \left| \int_{0}^{1} (u^{n} -u) \right| ^{2} |(a\bu_{x} )_{x} |^{2}\leq C \|(u^{n} ,h^{n} )-(u,h)\| _{E}^{2}   
	\intq \rho _{\ast}^{2} |(a\bu_{x} )_{x} |^{2} \\
	&\leq C \|(u^{n} ,h^{n} )-(u,h)\| _{E}^{2}  
	\|(\bu ,\bh )\| _{E}^{2}=C \|(u^{n} ,h^{n} )-(u,h)\| _{E}^{2} \to 0
	\end{align*} 
	and
	\begin{align*}
	\displaystyle
	I _{4}&=\left( \intq \rho _{0}^{2} |\bu |^{2} |D_{3} f(t,x,u^{n} ) -D_{3} f(t,x,u) |^{2} \right) ^{\frac{1}{2}}
	\left( \intq \rho _{0}^{2} |\bu |^{2} |D_{3} f(t,x,u^{n} )-D_{3} f(t,x,u) |^{2} \right) ^{\frac{1}{2}} \\
	&\leq C\left( \intq \rho _{0}^{2} |\bu |^{2} |D_{3} f(t,x,u^{n} ) -D_{3} f(t,x,u) |^{2} \right) ^{\frac{1}{2}}
	\left( \intq \rho _{0}^{2} |\bu |^{2} \right) ^{\frac{1}{2}} \\
	&\leq C\left( \intq \rho _{0}^{2} |\bu |^{2} |D_{3} f(t,x,u^{n} ) -D_{3} f(t,x,u) |^{2} \right) ^{\frac{1}{2}}
	\|(\bu ,\bh)\| _{E} \\
	&= C\left( \intq \rho _{0}^{2} |\bu |^{2} |D_{3} f(t,x,u^{n} ) -D_{3} f(t,x,u) |^{2} \right) ^{\frac{1}{2}}\to 0,
	\end{align*}
	as $n\to +\infty$.  It concludes the proof of Claim 2 as well as this lemma.
\end{proof}

\begin{lem} \label{lema3.5} 
	$H'(0,0)\in \mathcal{L}(E;F)$ is onto. 
\end{lem}

\begin{proof}
	Take $(g,u_0)\in F$. According to Theorem \ref{prop2.4}, there exists $(u,h)\in E$ that solves (\ref{prob3}). In other words,
	\begin{align*}
	H'(0,0)(u,h) & =(H_1'(0,0)(u,h),H_2'(0,0)(u,h))\\
	&= (u_t-(a(x)u_x)_x+c(t,x)u-h\chi_w,u(\cdot,0))\\
	& =(g,u_0).
	\end{align*}
	It completes the proof of this lemma.
\end{proof}

At this point, we are ready to prove our main result.

\begin{proof}[Proof of Theorem \ref{th_1.1}]
	We have proved through Lemmas \ref{lema3.3}, \ref{lema3.4} and \ref{lema3.5} that $H$ satisfies the hypothesis of Liusternik's Theorem. As a consequence, there exist $\ep>0$ and a right inverse mapping $\widetilde{H}:B_\ep(0)\subset F\to E$ of $H$. In particular, if $u_0\in H^1_a$ and $\|u_0\|_{H^1_a}< \ep$ then, $(u,h):=\widetilde{H}(0,u_0)$ solves the system \eqref{prob1}. Finally, the condition $\intq\rz^2u^2<+\infty$ yields
	$$u(T,x)=0, \mbox{ for all } x\in [0,1].$$
\end{proof}

\begin{rem} 
	
	Theorem \eqref{th_1.1} is still valid considering  the nonlocal terms 
	$$\into u^2 \mbox{ or } \into a(x)u_x^2,$$ 
	instead of $\into u$. It is true because we can adapt Lemma \ref{lema3.1} by using the Sobolev embedding $W^{1,1}(0,T)\hookrightarrow C(0,T)$.
\end{rem}

\newpage

\appendix

\section{Proof of Proposition \ref{cor_hat}} \label{appendix}
\noindent

In order to prove Proposition \ref{cor_hat}, we start considering the system
\begin{equation}\lnum \label{pbA1}
\left\{\begin{array}{ll}
v_t+\left(a\left(x\right)v_x \right)_x=h(t,x), & (t,x)\in \dom, \\
v(t,1)=v(t,0)=0 & t\in (0,T).\\
\end{array}\right.   \end{equation}

\begin{prop} \label{propA1}
	There exist $C>0$ and $\lambda_0,s_0>0$ such that every solution  $v$ of (\ref{pbA1}) satisfies, for all $s\geq s_0$ and $\lambda\geq \lambda_0$, 
	\begin{equation}
	\intq e^{2s\varphi}\left((s\lambda)\sigma av_x^2+(s\lambda)^2\sigma^2v^2 \right) \leq C\left(\intq e^{2s\varphi}|h|^2\   +(\lambda s)^3\intw e^{2s\varphi}\sigma^3v^2\   \right)
	\end{equation} 
	
\end{prop}

The proof of Proposition \ref{propA1} relies on the change of variables $w=e^{s\varphi}v$. Notice that
\begin{align*}
& v_t=e^{-s\varphi}(-s\varphi_tw+w_t),\\
& (av_x)_x=e^{-s\varphi}(s^2\varphi^2_xaw-s(a\varphi_x)_xw-2sa\varphi_xw_x+(aw_x)_x).
\end{align*} 
Then, from \eqref{pbA1}, we obtain 
$$
\begin{cases}
L^+w+L^-w=e^{s\varphi}h,  & (t,x)\in \dom, \\
w(t,1)=w(t,0)=0, & t\in (0,T).\\
w(x,0)=w(x,T)=0, & x\in (0,1),
\end{cases}$$
where
$$L^+w:=-s\varphi_t w+s^2\varphi_x^2aw+(aw_x)_x,$$
$$L^-w:=w_t-s(a\varphi_x)_xw-2sa\varphi_xw_x.$$

In this way,
$$\|L^+w\|^2+\|L^-w\|^2+2(L^+w,L^-w)=\|e^{s\varphi}h\|^2,$$
where $\|\cdot\|$ and $(\cdot,\cdot)$ denote the norm and the inner product in $L^2(\dom)$, respectively.

From now on, we will prove  Lemmas \ref{A1}--\ref{A18}. The proof of Proposition \ref{propA1} will be a consequence of these lemmas.
\begin{lem}\label{lemA2}
	\begin{align*}
	(L^+w,L^-w) & =\frac{s}{2}\intq\varphi_{tt}w^2-2s^2\intq \varphi_{tx}a\varphi_xw^2 +s^3\intq a\varphi_x(a\varphi_x^2)_xw^2\\
	&+s\intq(a\varphi_x)_{xx}aww_x
	+ 2s\intq(a\varphi_x)_xaw_x^2 -s\intq a\varphi_xa_xw_x^2 -s\int_0^T(a^2\varphi_xw_x^2)\big\vert_{x=0}^{x=1}
	\end{align*}
\end{lem}
\begin{proof}
	From the definition of $L^+w$ and $L^-w$ we have
	\begin{align*}
	(L^+w,L^-w) = & \intq (-s\varphi_tw+s^2\varphi_x^2aw+(aw_x)_x)w_t +s^2\intq\varphi_tw((a\varphi_x)_xw+2a\varphi_xw_x)\\
	& -s^3\intq \varphi^2_xaw((a\varphi_x)_xw+2a\varphi_xw_x) -s\intq(aw_x)_x((a\varphi_x)_xw+2a\varphi_xw_x)\\
	\phantom{(L^+w,L^-w)} = & I_1+I_2+I_3+I_4.
	\end{align*}
	
	Integrating by parts, we obtain
	$$I_1=\frac{s}{2}\intq(\varphi_{tt}-2s\varphi_x\varphi_{xt})w^2,$$
	$$I_2=-s^2\intq \varphi_{tx}a\varphi_xw^2,$$
	$$I_3=s^3\intq a\varphi_x(a\varphi_x^2)_xw^2$$
	and
	\begin{equation*} I_4=s\intq(a\varphi_x)_{xx}aww_x+2s\intq(a\varphi_x)_xaw_x^2
	-s\intq(a\varphi_x)a_xw_x^2-s\int_0^T(a^2\varphi_xw_x^2)\big\vert_{x=0}^{x=1},
	\end{equation*}
	which imply the desired result.

\end{proof}
\begin{lem}\label{A1}
	$\dps-s\int_0^Ta^2\varphi_xw_x^2\big\vert_{x=0}^{x=1}\geq 0$
\end{lem}
\begin{proof} Since $\psi'(x)=x/a$, if $x\in [0,\alpha')$ and $\psi'(x)=-x/a$, if $x\in (\beta'1]$, we have
	\begin{align*} 
	-s\int_0^Ta^2\varphi_xw_x^2\big\vert_{x=0}^{x=1}=-s\lambda\int_0^Ta^2\psi'\sigma w_x^2\big\vert_{x=0}^{x=1}\geq 0.
	\end{align*}
	
\end{proof}

\begin{lem}\label{A2}
	
	\begin{equation*}
	s^3\intq a\varphi_x(a\varphi_x^2)_xw^2\geq C \lambda^4s^3\intq a^2|\psi'|^4\sigma^3w^2-Cs^3\lambda^3\intwl \sigma^3w^2
	+s^3\lambda^3\inta\frac{x^2}{a}\sigma^3w^2
	\end{equation*}
\end{lem}
\begin{proof} Firstly, we observe that
	\begin{align*}
	s^3\intq a\varphi_x(a\varphi_x^2)_xw^2 &= s^3\lambda^3\intq a\psi'(a(\psi')^2)_x\sigma^3w^2 +2s^3\lambda^4\intq a^2(\psi')^4\sigma^3w^2\\
	& = I_1+I_2.
	\end{align*}

	We can see that
	$$a\psi'(a(\psi')^2)_x=\begin{cases}
	\phantom{-}\frac{x^2}{a^2}(2a-xa'), & x\in (0,\alpha')\\
	-\frac{x^2}{a^2}(2a-xa'), & x\in (\beta',1),
	\end{cases}$$
	and \eqref{prop_a} implies $2a-xa'\geq (2-K)a>a$. Hence,
	\begin{align*}
	I_1 &=s^3\lambda^3\inta a\psi'(a(\psi')^2)_x\sigma^3w^2+s^3\lambda^3\intwl a\psi'(a(\psi')^2)_x\sigma^3w^2+s^3\lambda^3\intb a\psi'(a(\psi')^2)_x\sigma^3w^2\\
	& \geq s^3\lambda^3\inta \frac{x^2}{a}\sigma^3w^2-Cs^3\lambda^3\intwl\sigma^3w^2
	-Cs^3\lambda^3\intq a^2|\psi'|^4\sigma^3w^2
	\end{align*}
	
	We just sum $I_1$ and $I_2$,  and take $\lambda_0$ large  enough to obtain the desired inequality.
	
\end{proof}

\begin{lem}\label{A3}
	\noindent 
	
	\begin{equation*}\dps 2s\intq(a\varphi_x)_xaw^2_x\geq -C\intwl \sigma w_x^2+Cs\lambda^2\intq a^2(\psi')^2\sigma w_x^2
	+2s\lambda\inta a\sigma w_x^2
	\end{equation*}
\end{lem}

\begin{proof} Observe that
	\begin{equation}\label{ast}
	2s\intq(a\varphi_x)_xw^2_x=2s\intq\lambda(a\psi')_xa\sigma w_x^2
	+2s\lambda^2\intq a^2(\psi')^2\sigma w_x^2
	\end{equation}
	
	Proceeding as in lemma before, we split the first integral over the intervals $[0,\alpha'], \omega'$ and $[\beta',1]$. Since $a^2(\psi')^2\geq Ca$ in  $[\beta',1]$ we can add the integral over $[\beta',1]$ to the last integral of \eqref{ast}, which gives us the result.
	
\end{proof}

\begin{lem}\label{A4}
	\begin{equation*}\dps-2s^2\intq\varphi_{tx}a\varphi_xw^2\geq -Cs^2\lambda^2\left(\inta \frac{x^2}{a}\sigma^3w^2+\intwl\sigma^3w^2\right.
	+\left.\intq a^2|\psi'|^4\sigma^3w^2\right)
	\end{equation*}
\end{lem}

\begin{proof} First of all,
	\begin{equation*}
	\left|2s^2\intq \varphi_{tx}a\varphi_xw^2\right|\leq 2s^2\lambda^2\intq a|\psi'|^2|\theta \theta'|\eta^2 w^2
	\leq Cs^2\lambda^2\intq a|\psi'|^2\sigma^3 w^2
	\end{equation*}
	
	As before, we split the last integral over the   intervals $[0,\alpha'], \omega'$ and $[\beta',1]$. The result comes from the boundedness of $a|\psi'|^2$  in $\omega'$ and from relations  $\psi'= x/a$ in $[0,\alpha']$ and $a|\psi'|^2\leq Ca^2|\psi'|^4$ in $[b',1]$.
	
\end{proof}

\begin{lem}\label{A5}
	$$-s\intq a\varphi_x a_xw_x^2\geq -K\lambda s \inta a\sigma w_x^2-c\lambda s \intwl \sigma w_x^2$$
\end{lem}

\begin{proof} In fact, from the definition of $\psi$, we obtain
	\begin{align*} 
	-s\intq a\varphi_x a_xw_x^2& =-s\lambda \intq aa_x\psi'\sigma w_x^2\\
	& \geq -Ks\lambda \inta a\sigma w_x^2-C\lambda s\intwl \sigma w_x^2,
	\end{align*}
	where we proceeded as in the proof of Lemma \ref{A4}.
\end{proof}

\begin{lem}\label{A9}
	\begin{align*}
	s\intq (a\varphi_x)_{xx}aw_xw \geq &  -Cs^2\lambda^4 \intq a^2|\psi'|^4 \sigma^3w^2-C\lambda^2 \intq a^2|\psi'|^2\sigma w_x^2 -Cs^2\lambda^3 \intwl \sigma^3w^2\\
	&-C\lambda\intwl \sigma w_x^2
	-Cs^2\lambda^3\inta \frac{x^2}{a}\sigma^3w^2-C\lambda \inta a \sigma w_x^2
	\end{align*}
\end{lem}

\begin{proof}
	\begin{align*}
	s\intq (a\varphi_x)_{xx}aw_xw &  =s\lambda\intwl (a\psi')_{xx}a\sigma w_xw +2s\lambda^2\intq (a\psi')_{x}\psi'a\sigma w_xw \\
	&  +s\lambda^2\intq a^2\psi'\psi''\sigma w_xw  +s\lambda^3\intq a^2(\psi')^3\sigma w_xw\\
	& =I_1+I_2+I_3+I_4.
	\end{align*}
	
	The inequality will be obtained by estimating each one of these fours integrals. For $I_1$, we have
	\begin{align*}
	|I_1|& =\left|s\lambda\intwl (a\psi')_{xx}a\sigma w_xw\right|\leq Cs\lambda\intwl \sigma^2 |w_xw|\\
	& =Cs\lambda\intwl \sigma^{3/2}|w|\sigma^{1/2} |w_x|\leq Cs\lambda\intwl \sigma^3 w^2+Cs\lambda\intwl \sigma w_x^2.
	\end{align*}
	
	For $I_2$, we use the facts $\sigma\leq C\sigma^2$ and $x\leq Ca^2|\psi'|^3$ in $[\beta',1]$ to obtain
	\begin{align*}
	|I_2| & \leq Cs\lambda^2\inta x\sigma^2 |w w_x|+ Cs\lambda^2\intwl \sigma^2 |w w_x|
	+Cs\lambda^2\intb a^2|\psi'|^3\sigma^2 |ww_x|\\
	&\leq  C\inta \left(\frac{x}{\sqrt{a}}\sigma^{3/2}\lambda^{3/2}s|w|\right)(\sqrt{a}\sigma^{1/2}\lambda^{1/2}|w_x|) + C\intwl \left(\sigma^{3/2}\lambda^{3/2}s|w|\right)(\sigma^{1/2}\lambda^{1/2}|w_x|)\\
	& + C\intb \left(s\sigma^{3/2}\lambda^{2} a|\psi'|^2|w|\right)(\sigma^{1/2}a|\psi'||w_x|)\\
	& \leq C\lambda^3s^2\inta \frac{x^2}{a}\sigma^3w^2 +C\lambda \inta a\sigma w_x^2 + C\lambda^3s^2\intwl \sigma^3w^2 +C\lambda\intwl \sigma w_x^2\\
	&+ Cs^2\lambda^4 \intq a^2 |\psi'|^4\sigma^3w^2 +C\intq a^2|\psi'|^2\sigma w_x^2
	\end{align*}
	
	For $I_3$, recalling the definition of $\psi$, se observe that
	\begin{align*}
	|I_3| & \leq s\lambda^2 \inta \left|x\left(\frac{a-xa'}{a}\right)\right|\sigma |ww_x|+Cs\lambda^2\intwl \sigma^2 |ww_x|+ s\lambda^2 \intb  \left|x\left(\frac{a-xa'}{a}\right)\right|\sigma |ww_x|\\
	& \leq Cs\lambda^2\inta x\sigma^2 |ww_x|+Cs\lambda^2\intwl \sigma^2 |ww_x| + Cs\lambda^2\intb x\sigma |ww_x|.
	\end{align*} 
	So, we get the same estimate for $I_2$. Finally,
	\begin{equation*}
	|I_4| \leq \intq |s\lambda^2 a(\psi')^2\sigma^{3/2}w| |\lambda a \psi'\sigma^{1/2}w_x|\leq  Cs^2\lambda^4\intq a^2|\psi'|^4\sigma^3w^2+C\lambda^2\intq a^2|\psi'|\sigma w_x^2,
	\end{equation*}
	and the proof is complete.
\end{proof}

\begin{lem}\label{A10}
	\begin{align*}
	\frac{s}{2}\intq \varphi_{tt}w^2\geq & -Cs^{1/2}\inta \sigma a w_x^2-Cs^2\lambda^2\inta \frac{x^2}{a}\sigma^3w^2 -Cs^{1/2}\intwl \sigma w_x^2\\ 
	&-C\lambda^2s^2\intwl \sigma^3w^2-Cs^{1/2}\intq a^2|\psi'|^2\sigma w_x^2-Cs^2\lambda^2\intq a^2|\psi'|^4\sigma^3w^2
	\end{align*}
\end{lem}

\begin{proof}
	Firstly, we observe that $|\varphi_{tt}| \leq C\sigma^{3/2} $. Then we apply  Hardy-Poincar\'e inequality, to take
	\begin{align*}
	\left|\frac{s}{2}\intq \varphi_{tt}w^2\right|&\leq Cs\intq \sigma^{3/2}w^2\leq  \intq\left(s^{1/4}\sigma^{1/2}\frac{\sqrt{a}}{x}w\right)\left(s^{3/4}\sigma\frac{x}{\sqrt{a}}w\right)\\
	& \leq C\intq s^{1/2}\sigma\frac{a}{x^2}w^2+C\intq s^{3/2}\sigma^2\frac{x^2}{a}w^2\\
	& \leq C\intq s^{1/2}\sigma aw^2_x+C\lambda^2s^{2}\intq\sigma^3\frac{x^2}{a}w^2\\
	\end{align*}
	
	Again, the two last intervals can be decomposed in  $[0,\alpha']$, $\omega'$ and $[\beta',1]$. At this point,  relations 
	$$\ a\leq Ca^2|\psi'|^2  \mbox{ and } \frac{x^2}{a}\leq C a^2|\psi'|^4, \mbox{ in } [\beta',1],$$
	give us the result.
	
\end{proof}

\begin{lem}\label{A17}
	\begin{multline*}
	s^3\lambda^3 \inta \frac{x^2}{a}\sigma^3w^2+s\lambda \inta \sigma a w_x^2 +
	s^3\lambda^4\intq a^2|\psi'|^4\sigma^3w^2+s\lambda^2\intq a^2|\psi'|^2\sigma w_x^2\\
	\leq C\left(\intq e^{2s\varphi}|h|^2+s^3\lambda^3\intwl\sigma^3w^2+\lambda s \intwl \sigma w_x^2\right)
	\end{multline*}
\end{lem}

\begin{proof}
	From Lemmas \ref{lemA2}-\ref{A10}, we have
	\begin{multline*}
	(L^+w,L^-w)\geq  C\Bigg( s^3\lambda^3 \inta \frac{x^2}{a}\sigma^3w^2+s\lambda \inta \sigma a w_x^2 +
	\lambda^4s^3\intq a^2|\psi'|^4\sigma^3 w^2\\
	+s\lambda^2\intq a^2|\psi'|^2\sigma w_x^2 -s^3\lambda^3\intwl \sigma^3w^3-\lambda s \intwl \sigma w_x^2\Bigg).
	\end{multline*}
	
	Hence, 
	
	\begin{align*}
	& C\Bigg( s^3\lambda^3 \inta \frac{x^2}{a}\sigma^3w^2+s\lambda \inta \sigma a w_x^2
	+ \lambda^4s^3\intq a^2|\psi'|^4\sigma^3 w^2\\ 
	&+s\lambda^2\intq a^2|\psi'|^2\sigma w_x^2
	-s^3\lambda^3\intwl \sigma^3w^3-\lambda s \intwl \sigma w_x^2\Bigg)\\
	& \leq  \n{L^+w}^2+\n{L^-w}^2+2(L^+w,L^-w) \leq \n{e^{s\varphi}h}^2,
	\end{align*}
	following the result.
\end{proof}

Now, we intend to prove a suitable inequality which will imply Proposition \ref{propA1}. In order to do that, we recall that $v=e^{-s\varphi}w$.

\begin{lem}\label{A18}
	\begin{multline*}
	s^3\lambda^3 \inta e^{2s\varphi} \frac{x^2}{a} \sigma^3 v^2+s\lambda \inta e^{2s\varphi} \sigma av_x^2 
\\	+s^3\lambda^4 \intq e^{2s\varphi} a^2|\psi'|^4\sigma^3v^2+s\lambda^2\intq e^{2s\varphi}a^2|\psi'|^2\sigma v_x^2\\
	\leq C\left(\intq e^{2s\varphi} |h|^2 +\lambda^3s^3 \intw e^{2s\varphi}\sigma^3v^2\right)
	\end{multline*}
\end{lem}

\begin{proof}
	Since $v=e^{-s\varphi}w$, we have
	\begin{align*}
	&  e^{s\varphi}v_x=-s\lambda \psi' \sigma w+w_x 
	\end{align*}
	which implies
	\begin{align*}
	e^{2s\varphi}s\lambda^2|\psi'|^2a^2\sigma v_x^2
	& =(s\lambda^2|\psi'|^2a^2\sigma)e^{2s\varphi}v_x^2 \leq C(s\lambda^2|\psi'|^2a^2\sigma)(s^2\lambda^2|\psi'|^2\sigma^2w^2+w_x^2)\\
	& \leq C(s^3\lambda^4|\psi'|^4\sigma^3a^2w^2+s\lambda^2|\psi'|^2a^2\sigma w_x^2)
	\end{align*}
	Besides that,
	\begin{align*}
	& w_x=s\varphi_xe^{s\varphi}v+e^{s\varphi}v_x\Rightarrow w_x^2\leq C(s^2\lambda^2|\psi'|^2\sigma^2e^{2s\varphi}v^2+e^{2s\varphi}v_x^2)\nonumber\\
	\Rightarrow & w_x^2\leq C(s^2\lambda^2\sigma^2e^{2s\varphi}v^2+e^{2s\varphi}av_x^2), \mbox { in } \omega'
	\end{align*}
	
	Hence, from Lemma \ref{A17}, we get
	\begin{align}\label{A12}
	& s^3\lambda^3 \inta e^{2s\varphi} \frac{x^2}{a} \sigma^3  v^2+s\lambda \inta e^{2s\varphi} \sigma av_x^2 +	s^3\lambda^4 \intq e^{2s\varphi} a^2|\psi'|^4\sigma^3v^2\nonumber\\
	&\phantom{s^3\lambda^3 \inta e^{2s\varphi} \frac{x^2}{a} \sigma^3  v^2+s\lambda \inta e^{2s\varphi} \sigma av_x^2 +	s^3\lambda^4 \intq} +s\lambda^2\intq e^{2s\varphi}a^2|\psi'|^2\sigma v_x^2\nonumber\\
	&\leq C\bigg(  s^3\lambda^3 \inta \frac{x^2}{a} \sigma^3  w^2+s\lambda \inta  \sigma aw_x^2  +	s^3\lambda^4 \intq a^2|\psi'|^4\sigma^3w^2+s\lambda^2\intq a^2|\psi'|^2\sigma w_x^2\nonumber\\
	&\leq C\left(\intq e^{2s\varphi}|h|^2+s^3\lambda^3\intwl\sigma^3w^2+\lambda s \intwl \sigma w_x^2\right)\nonumber\\
	& \leq C\left(\intq e^{2s\varphi}|h|^2+s^3\lambda^3\intwl e^{2s\varphi}\sigma^3v^2+\lambda s \intwl e^{2s\varphi}  \sigma av_x^2\right)
	\end{align}
	
	To complete the proof we will estimate the last integral of \eqref{A12}. Firstly, let us take $\chi\in C_0^\infty(\omega)$  such that $0\leq \chi \leq 1$ and $\chi\equiv 1$ in $\omega'$. Multiplying equation in \eqref{pbA1} by $\lambda s e^{2s\varphi}\sigma v \chi$ and integrating over $\dom$, we obtain
	\begin{equation}\label{A13}
	\lambda s\intq e^{2s\varphi}\sigma v v_t\chi+\lambda s\intq e^{2s\varphi} \sigma (av_x)_xv\chi 
	=\lambda s \intq e^{2s\varphi}\sigma hv\chi.
	\end{equation}
	
	We can see that
	\begin{multline}\label{A14}
	\left|\intq e^{2s\varphi}\sigma v v_t\chi\right|=\left|\frac{1}{2}\intq e^{2s\varphi}\sigma \frac{d}{dt}v^2\chi\right|=\left|-\frac{1}{2}\intq (e^{2s\varphi}\sigma \chi)_{_t}v^2\right| \\
	=\left|-\frac{1}{2}\intq \chi e^{2s\varphi}(2s\varphi_t\sigma+\sigma_t)v^2\right|\leq Cs\intw e^{2s\varphi} \sigma^3v^2.
	\end{multline}
	
	And, analogously,
	\begin{align*}
	\intq e^{2s\varphi}\sigma (av_x)_x v \chi=-\intq e^{2s\varphi}\sigma a v_x^2 \chi -\intq (e^{2s\varphi}\sigma \chi)_xav_xv.
	\end{align*}
	Since $\varphi_x\leq C\sigma$ and $\sigma_x\leq C\sigma$ in $\domw$, we get
	\begin{equation}\label{A15}
	\left|\intq (e^{2s\varphi}\sigma \chi)_xav_xv\right|\leq C\intw  e^{2s\varphi} \sigma^2 |av_x||v|.
	\end{equation}
	
	Now, from \eqref{A13}-\eqref{A15} we obtain
	\begin{align*}
	& \lambda s\intw e^{2s\varphi} \sigma a v_x^2=\lambda s\intq e^{2s\varphi} \sigma a v_x^2\chi\\
	& \leq \left|-\lambda s \intq \wei\sigma (a v_x)_xv\chi -\lambda s \intq (\wei \sigma \chi)_xav_xv\right|\\
	& \leq \lambda s \intq \wei \sigma |vv_t|\chi +\lambda s \intq \wei \sigma |hv|\chi +\lambda s \intq |(\wei \sigma\chi)_x||av_xv| \\
	& \leq C\lambda s^2\intw e^{2s\varphi} \sigma^3v^2 + \lambda s \intw (e^{s\varphi} h)(e^{s\varphi}\sigma v)+ C\lambda s\intw  e^{2s\varphi} \sigma^2 |av_x||v|\\
	& \leq C\lambda^3 s^3\intw e^{2s\varphi} \sigma^3v^2 + \frac{1}{2}\lambda s \intw e^{2s\varphi} h^2+\frac{1}{2}\lambda s \intw e^{2s\varphi} \sigma^2 v^2\\
	& \phantom{\lambda s \intq \wei \sigma |vv_t|\chi }+ C\lambda s\intw  (e^{s\varphi} \sigma^{1/2}a^{1/2}|v_x| )(e^{s\varphi} a^{1/2}\sigma^{3/2}|v|)\\
	& \leq C\lambda^3 s^3\intw e^{2s\varphi} \sigma^3v^2 + C \intw e^{2s\varphi} h^2\\
	& \phantom{\lambda s \intq \wei \sigma |vv_t|\chi }+ \ep C\lambda s\intw  e^{2s\varphi} \sigma a v_x^2+C_\ep \intw e^{2s\varphi} a\sigma^3 v^2\\
	& \leq C\lambda^3 s^3\intw e^{2s\varphi} \sigma^3v^2 + C \intw e^{2s\varphi} h^2+ \ep C\lambda s\intw  e^{2s\varphi} \sigma a v_x^2.
	\end{align*}
	Hence, taking $\ep=1/2C$, we get
	\begin{equation*}
	\lambda s\intw e^{2s\varphi} \sigma a v_x^2\leq  C\left(  \intw e^{2s\varphi} h^2+\lambda^3 s^3\intw e^{2s\varphi} \sigma^3v^2 \right).
	\end{equation*}
	It last inequality combined with \eqref{A12} completes the proof.
	
\end{proof}

Now we are ready to prove Proposition \ref{propA1}.

\begin{proof}[Proof of Proposition \ref{propA1}]
	\noindent 
	
	Firstly, we observe that
	\begin{align*}
	& \lambda^2s^2\intq \wei \sigma^2v^2=\lambda^2s^2\intq \sigma^2w^2= \intq \left(\lambda^{3/2}s^{3/2}\sigma^{3/2}\frac{x}{\sqrt{a}}w\right)\left(\lambda^{1/2}s^{1/2}\sigma^{1/2}\frac{\sqrt{a}}{x}w\right)\\
	& \leq \frac{s^3\lambda^3}{2}\intq \sigma^3\frac{x^2}{a}w^2+\frac{s\lambda}{2}\intq \sigma\frac{a}{x^2}w^2=I_1+I_2.
	\end{align*}
	
	Now, let us estimate $I_1$ and $I_2$ taking into account the terms of the inequality given by Lemma \ref{A17}.
	
	Splitting $I_1$ over the   intervals $[0,\alpha'], \omega'$ and $[\beta',1]$, and taking into account that $x^2/a$ is bounded in $\omega'$ and $x^2/a\leq a^2|\psi'|^2$ in $[b',1]$, we use Lemma \ref{A17} to obtain that
	\begin{align*}
	I_1 & \leq s^3\lambda^3\inta \sigma^3\frac{x^2}{a}w^2+Cs^3\lambda^3\intwl \sigma^3 w^2+Cs^3\lambda^4\intq \sigma^3 a^2|\phi'|^4w^2\\
	& \leq C\left(\intq \wei h^2+s^3\lambda^3\intwl \sigma^3 w^2+\lambda s\intwl \sigma w_x^2\right)
	\end{align*}
	
	In addition, we apply Hardy-Poincar\'e inequality to estimate $I_2$, as following
	\begin{align*}
	& I_2= \frac{s\lambda}{2}\intq \frac{a}{x^2}(\sigma^{1/2}w)^2\leq Cs\lambda\intq a(\sigma^{1/2}w)_x^2\\
	& =Cs\lambda \intq a\left(\frac{1}{2}\sigma^{-1/2}\sigma_xw+ \sigma^{1/2}w_x\right)^2\\
	& \leq Cs\lambda \intq a\sigma^{-1}\sigma_x^2w^2+Cs\lambda \intq a\sigma w_x^2\\
	& \leq Cs\lambda^3 \intq a|\psi'|^2\sigma w^2+Cs\lambda \intq a\sigma w_x^2\\
	& \leq Cs^3\lambda^3 \inta \frac{x^2}{a}\sigma^3 w^2+Cs^3\lambda^3 \intwl \sigma^3 w^2+Cs^3\lambda^4 \intb a^2|\psi'|^4\sigma^3 w^2\\
	&\phantom{\leq}+Cs\lambda \inta a\sigma w_x^2+Cs\lambda \intwl \sigma w_x^2+Cs\lambda^2 \intb a^2|\psi'|^2\sigma w_x^2\\
	& \leq C\left(\intq \wei h^2+s^3\lambda^3\intwl \sigma^3 w^2+\lambda s\intwl \sigma w_x^2\right)
	\end{align*}
	
	Hence,
	\begin{equation*}
	\lambda^2s^2\intq \wei \sigma^2v^2\leq I_1+I_2 \leq C\left(\intq \wei h^2+s^3\lambda^3\intwl \sigma^3 w^2+\lambda s\intwl \sigma w_x^2\right).
	\end{equation*}
	Proceeding exactly as in the proof of Lemma \ref{A18}, we achieve
	\begin{equation*}\label{A20}
	\lambda^2s^2\intq \wei \sigma^2v^2\leq  C\left(  \intw e^{2s\varphi} h^2+\lambda^3 s^3\intw e^{2s\varphi} \sigma^3v^2 \right),
	\end{equation*}
	and the result given by Lemma \ref{A18} gives us
	\begin{align*}
	s\lambda \intq \wei a\sigma v_x^2 &\leq s\lambda \inta \wei a\sigma v_x^2+s\lambda \intwl \wei a\sigma v_x^2	+s\lambda \intb \wei a^2|\psi'|^2\sigma v_x^2\\
	&\leq  C\left(  \intw e^{2s\varphi} h^2+\lambda^3 s^3\intw e^{2s\varphi} \sigma^3v^2 \right).
	\end{align*}
	
	Therefore, this last two estimates conclude the proof of Proposition \ref{propA1}.
	
\end{proof}

\begin{proof}[Proof of Propostion \ref{cor_hat}]
	If $v$ is a solution of \eqref{prob2}, then $v$ is also a solution of \eqref{pbA1} with $h=F+cv$. In this case, applying Propostion \ref{propA1}, there exist $C>0$, $\lambda_0>0$ and $s_0>0$ such that $v$ satisfies, for all $s\geq s_0$ and $\lambda\geq \lambda_0$, 
	\begin{equation} \label{A19}
	\intq e^{2s\varphi}\left((s\lambda)\sigma av_x^2+(s\lambda)^2\sigma^2v^2 \right) \\
	\leq C\left(\intq e^{2s\varphi}|h|^2\   +(\lambda s)^3\intw e^{2s\varphi}\sigma^3v^2\   \right).
	\end{equation} 
	
	Recalling that $c\in L^{\infty}(\dom)$, we can see that
	\begin{align*}
	\intq e^{2s\varphi}|h|^2& =\intq e^{2s\varphi}|F+cv|^2\\
	& \leq C\intq e^{2s\varphi}|F|^2+C\n{c}_{_\infty}^2\intq e^{2s\varphi}|v|^2\\
	& \leq C\intq e^{2s\varphi}|F|^2+C\intq e^{2s\varphi}\sigma^2|v|^2.
	\end{align*}
	
	Finally, taking $\lambda_0$ and $s_0$ large enough, the last integral can be absorbed by the left-hand side of \eqref{A19} which complete the proof.

\end{proof}

\bibliographystyle{acm}
\bibliography{references}

\end{document}